\documentclass[11pt]{article}

\usepackage[a4paper,margin=1in]{geometry}

\usepackage{amsmath,amssymb,amsthm}
\usepackage{mathtools}
\usepackage{xcolor}
\usepackage[normalem]{ulem}

\usepackage[numbers,sort&compress]{natbib}
\usepackage[colorlinks=true,
            linkcolor=blue,
            citecolor=blue,
            urlcolor=blue]{hyperref}

\usepackage{microtype}
\allowdisplaybreaks[4]

\theoremstyle{plain}
\newtheorem{theorem}{Theorem}

\newtheorem*{theorem2}{Theorem 2}
\newtheorem{lemma}[theorem]{Lemma}

\newtheorem{proposition}[theorem]{Proposition}

\theoremstyle{definition}

\theoremstyle{remark}

\newcommand{\one}{\mathbf1}
\newcommand{\trans}{^{\mathsf T}}
\newcommand{\R}{\mathbb R}
\newcommand{\Z}{\mathbb Z}
\newcommand{\C}{\mathbb C}

\title{Polynomial positivity cones for Coxeter roots and walks in trees}

\author{
Dongxiu Cai\thanks{
School of Mathematical Sciences, MOE-LSC, SHL-MAC, Shanghai Jiao Tong University,
800 Dongchuan Road, Shanghai 200240, P.R.~China.
Email: \texttt{diudiutse@sjtu.edu.cn}; \texttt{czb911@sjtu.edu.cn} (Zhenbo Chen).
}
\and
Zhenbo Chen\footnotemark[1]
\and
Jiasheng Zeng\thanks{
Department of Mathematics, Hong Kong University of Science and Technology,
Clear Water Bay, Kowloon, Hong Kong.
Email: \texttt{jzengbl@connect.ust.hk}.
}
\and
Xiao-Dong Zhang\footnotemark[1]\hspace{0.25em}\thanks{
Corresponding author. Email: \texttt{xiaodong@sjtu.edu.cn}.
}
}

\date{}

\begin{document}
\maketitle

\begin{abstract}
For a finite simple graph $G$ and an integer $k\ge0$, let $w_k(G)$ denote the number of walks of length $k$. We prove the conjecture of T\"aubig, Weihmann, Kosub, Hemmecke, and Mayr for every finite tree and determine all equality cases. If $T$ has $n\ge1$ vertices, then $n w_{k+1}(T)-2(n-1)w_k(T)\ge0$ for every $k\ge1$; for $n\ge3$, equality holds if and only if $T$ is a star and $k$ is even, whereas for $n=1$ or $n=2$, equality holds for every $k\ge1$. For non-Dynkin trees and even indices, the proof is based on a polynomial positivity cone associated with the adjacency operator of a finite graph and a positive real root of its simply-laced Coxeter system. For finite connected bipartite non-Dynkin graphs, we establish sufficient positivity conditions in terms of Coxeter orbits and inversion sets, and verify these conditions for indicator roots supported on connected induced subtrees. For non-Dynkin trees, this yields the rooted even-index inequality and, after summation, the corresponding global inequality. We also prove that if $G$ is a finite connected bipartite non-Dynkin simple graph, $\varnothing\ne U\subseteq V(G)$, and the subgraph of $G$ induced by $U$ is a tree, then $|U|w_{k+1}(G,U)-2(|U|-1)w_k(G,U)\ge0$ for every $k\ge0$, where $w_k(G,U)$ counts the length-$k$ walks in $G$ whose initial and terminal vertices lie in $U$; the intermediate vertices are unrestricted. The remaining even-index cases for finite Dynkin trees are handled by generating-function recurrences, while the odd-index cases follow from a spectral covariance identity.
\end{abstract}



\medskip
\section{Introduction}\label{sec:intro}

For a finite simple graph \(G=(V,E)\) and an integer \(k\ge0\), let
\(w_k(G)\) denote the number of walks of length \(k\), where a walk is
an ordered sequence \((v_0,\ldots,v_k)\) in which consecutive vertices
are adjacent. Vertices and edges may be repeated, and a walk is counted
separately from its reversal unless the two sequences coincide. Thus
$w_0(G)=|V|$ and $w_1(G)=2|E|$.
If \(A\) is the adjacency matrix of \(G\) and \(\one\) is the all-one
column vector indexed by \(V\), then
$w_k(G)=\one\trans A^k\one$.
Consequently, the walk sequence is governed by the spectrum of the
adjacency matrix. The use of adjacency spectra to enumerate walks is
classical; in particular, Harary and Schwenk related walk-generating
functions to characteristic polynomials and graph spectra
\cite{HararySchwenk79}.

Classical H\"older-type inequalities for powers of nonnegative
symmetric matrices yield log-convexity relations for suitable walk
subsequences \cite{BlakleyDixon66,ErdosSimonovits82}. Lagarias, Mazo,
Shepp, and McKay established further inequalities for graph walks
\cite{LagariasEtAl84}, and Dress and Gutman obtained related comparisons
for walk sequences \cite{DressGutman03}. T\"aubig and Weihmann
subsequently established weighted extensions for arbitrary nonnegative
vertex weights \cite{TaubigWeihmann14}.

The inequality
$w_1(G)w_k(G)\le w_0(G)w_{k+1}(G)$
can fail even for bipartite graphs and for graphs without cycles
\cite{TWKHM13}. For trees, T\"aubig, Weihmann, Kosub, Hemmecke, and
Mayr proved the cases \(k=2\) and \(k=4\), and formulated the following
conjecture \cite[Conjecture~26]{TWKHM13}: $w_0(T)w_{k+1}(T)-w_1(T)w_k(T)\ge0$
for every $k\ge1$.
For odd \(k\), the same inequality holds for every finite graph by the
Sandwich Theorem \cite[Theorem~4]{TWKHM13}. More recently, Cai, Zeng,
and Zhang proved the conjecture for several classes of trees, including
trees of diameter at most four, and proposed the corresponding equality
characterization \cite{CZZ26}. The purpose of this article is to prove
the conjecture for every finite tree and to determine all equality
cases, as we show in the following
theorem.

\begin{theorem}\label{thm:main}
Let \(T\) be a finite nonempty simple tree on \(n\) vertices. For every
integer \(k\ge1\),
\begin{equation}\label{eq:main}
n w_{k+1}(T)-2(n-1)w_k(T)\ge0.
\end{equation}
If \(n\ge3\), equality holds if and only if \(T\) is a star and \(k\) is
even. If \(n=1\) or \(n=2\), equality holds for every \(k\ge1\).
\end{theorem}

Since a tree on \(n\) vertices has \(n-1\) edges,
\eqref{eq:main} is exactly the conjectured inequality. The case \(k=0\),
included in the convention of \cite{TWKHM13}, is an identity for every
tree. The equality statement shows that, for trees with at least three
vertices, equality occurs precisely for stars and positive even \(k\).

The proof is organized around a more general positivity problem. Let $A$ be the adjacency matrix of a finite connected bipartite simple graph and put $C=2I-A$. This matrix is the generalized Cartan matrix of the associated simply-laced Coxeter system. We use its standard geometric representation and the corresponding real roots; see, for example, \cite[Sections~5.3--5.7]{Humphreys90} and \cite{Fu12}. For a positive real root $\alpha$, let $s_\alpha$ be the corresponding reflection and put $L_\alpha=I+s_\alpha$. We study the polynomial positivity cone $\mathcal P_\alpha(A) = \{q\in\mathbb R[x]_{\mathrm{even}}: L_\alpha q(A)\alpha\ge0\}$, where $\mathbb R[x]_{\mathrm{even}}$ denotes the space of real even polynomials and vector inequalities are understood coordinatewise. The root-positivity criterion stated as Theorem~\ref{thm:root-positivity} below gives  explicit sufficient conditions for families of even polynomials to lie in $\mathcal P_\alpha(A)$. We emphasize that these are sufficient inclusion conditions; we do not claim to characterize the entire cone $\mathcal P_\alpha(A)$.

To state Theorem~\ref{thm:root-positivity}, we first introduce the
root-theoretic and polynomial notation appearing in its formulation. Let $\Phi^+$ and $\Phi^-$ denote the positive and negative real roots, and let $\mathcal N(s_\alpha)=\{\beta\in\Phi^+:s_\alpha\beta\in\Phi^-\}$. A connected simple graph is called finite Dynkin if it is of type $A_n$, $D_n$, $E_6$, $E_7$, or $E_8$, and non-Dynkin otherwise.

For each vertex $v$, let $s_v$ denote the associated simple reflection. For a bipartition $V=X\sqcup Y$, put $R_X=\prod_{v\in X}s_v$, $R_Y=\prod_{v\in Y}s_v$, and $\Omega=R_XR_Y$; the factors within each product commute. Define $p_{-1}(x)=0$, $p_0(x)=1$, and $p_{j+1}(x)=xp_j(x)-p_{j-1}(x)$ for $j\ge0$. We write $\mathbb R_{\ge0}[x^2]$ for the cone of even polynomials with nonnegative coefficients.
With this notation, Theorem~\ref{thm:root-positivity} can
be stated as follows.

\begin{theorem}[Root-positivity criterion]\label{thm:root-positivity}
Let $G$ be a finite connected bipartite non-Dynkin simple graph with adjacency matrix $A$, and let $\alpha\in\Phi^+$. Suppose either that $\Omega^j\alpha\in\Phi^-$ for some $j\in\mathbb Z$, or that every $\beta\in\mathcal N(s_\alpha)$ satisfies $0\le\beta\le\alpha$. Then $L_\alpha p_{2m}(A)\alpha\ge0$ for every $m\ge0$. Consequently, $\operatorname{cone}\{p_0,p_2,p_4,\ldots\}\subseteq\mathcal P_\alpha(A)$ and $\mathbb R_{\ge0}[x^2]\subseteq\mathcal P_\alpha(A)$.
\end{theorem}

Theorem~\ref{thm:root-positivity} is the main structural result. It supplies the even-index inequality for non-Dynkin trees; the finite Dynkin cases and the odd-index inequality require the separate arguments described below. Together, these ingredients yield Theorem~\ref{thm:main}.

\subsection{Organization of the paper}

Section~\ref{sec:polynomial-cones} develops the Coxeter and polynomial tools used in the proof of Theorem~\ref{thm:root-positivity}. We first define the bilinear form, simple reflections, real roots, root signs, and reflections associated with arbitrary positive real roots. We then introduce the cone $\mathcal P_\alpha(A)$ and record the coordinate formula for $L_\alpha q(A)\alpha$. The polynomials $p_j$, defined by $p_{-1}(x)=0$, $p_0(x)=1$, and $p_{j+1}(x)=xp_j(x)-p_{j-1}(x)$, are related to the matrix Hilbert series of a preprojective algebra. The theorem of Etingof and Eu implies that $p_j(A)\ge0$ entrywise for every finite connected non-Dynkin graph \cite[Theorem~3.4.2]{EE07}. We also prove that every monomial $x^{2m}$ is a nonnegative linear combination of $p_0,p_2,\ldots,p_{2m}$.

Section~\ref{sec:root-positivity} proves the criterion stated as Theorem~\ref{thm:root-positivity} for positive real roots on finite connected bipartite non-Dynkin graphs. For a bipartition $V=X\sqcup Y$, we form the two commuting products of simple reflections $R_X$ and $R_Y$, and put $\Omega=R_XR_Y$. The key identities are $\Omega+\Omega^{-1}=A^2-2I$ and $\Omega^j+\Omega^{-j} = p_{2j}(A)-p_{2j-2}(A)$ for $j\ge1$. These identities yield a telescoping expansion of $p_{2m}(A)\alpha$.

The proof of Theorem~\ref{thm:root-positivity} has two parts. If some power of $\Omega$ sends $\alpha$ to a negative root, we follow the orbit until its first sign change. Reversing the preceding reflections produces a nonnegative matrix $H$ satisfying $He_v=\alpha$ for a suitable vertex $v$. For every $r\in V$ and $m\ge0$, the $r$th coordinate satisfies $\bigl(L_\alpha p_{2m}(A)\alpha\bigr)_r=2\sum_{u\in V\setminus\{v\}}H_{ru}\bigl(p_{2m}(A)\bigr)_{uv}$ and is therefore nonnegative. If the entire $\Omega$-orbit of $\alpha$ consists of positive roots, we instead use the hypothesis $\mathcal N(s_\alpha) \subseteq \{\beta\in\Phi^+:0\le\beta\le\alpha\}$. This gives a coordinate estimate for roots outside the inversion set. A positive Perron eigenvector then rules out the only remaining possible obstruction.

Section~\ref{sec:trees} applies Theorem~\ref{thm:root-positivity} to indicator vectors supported on connected induced trees in finite connected bipartite non-Dynkin graphs. For such a graph $T$ and every nonempty set $U\subseteq V(T)$ for which $T[U]$ is a tree, we show that its indicator vector $\one_U$ is a positive real root and that $\mathcal N(s_{\one_U})$ is the disjoint union of $\{\one_U\}$ and the sets $\{\one_{U_{e,1}},\one_{U_{e,2}}\}$ over all $e\in E(T[U])$, where $U_{e,1}$ and $U_{e,2}$ are the vertex sets of the two components of $T[U]-e$. In particular, every root in this inversion set is coordinatewise bounded above by $\one_U$, so Theorem~\ref{thm:root-positivity} applies. The walk interpretation converts this coordinatewise positivity statement into rooted even-index walk inequalities. When $T$ itself is a non-Dynkin tree, taking $U=V(T)$ yields the rooted even-index inequality, and summing it over all root vertices proves the even-index cases of Theorem~\ref{thm:main}.

Section~\ref{sec:finite} treats finite Dynkin trees, which are excluded from the non-Dynkin Hilbert-series argument. The proof uses rooted walk-generating functions and vertex-deletion recurrences to establish the required rooted inequalities for paths and the trees of types $D_n,E_6,E_7,E_8$. Resolvent identities of this type are classical \cite{Godsil92}; the particular recurrences required here are derived directly. The comparison of rooted walk counts is combined with the results of Cai, Zeng, and Zhang \cite{CZZ26}.

Finally, Section~\ref{sec:conclusion} completes the proof of Theorem~\ref{thm:main} and derives, as a further application, the $U$-restricted inequality. The first subsection treats the odd-index cases using the spectral covariance identity underlying the Sandwich Theorem \cite[Theorem~4]{TWKHM13}, determines their equality cases, and completes the equality analysis for positive even indices. The second subsection derives the global $U$-restricted walk inequality by combining polynomial positivity for even indices with the spectral covariance argument for odd indices.

\section{Coxeter reflections and polynomial positivity cones}
\label{sec:polynomial-cones}

Throughout, \(I\) denotes the identity matrix indexed by the relevant
index set; the index set will be clear from the context. For a vector \(x\), \(x_v\) denotes its coordinate at the index \(v\),
and \(e_v\) denotes the corresponding standard basis vector. For a
matrix \(M\), \(M_{uv}\) denotes its entry in row \(u\) and column \(v\).

\subsection{Reflections and real roots}\label{subsec:reflections}

Let \(G=(V,E)\) be a finite connected simple graph with adjacency matrix
\(A\). Set $C=2I-A$, define $\langle x,y\rangle=x\trans Cy$, and let
$Q(x)=\tfrac12\langle x,x\rangle$.
For each vertex $v$, define the simple reflection
$s_vx=x-\langle e_v,x\rangle e_v$.
It changes only the coordinate at $v$, replacing $x_v$ by
$\sum_{u\sim v}x_u-x_v$. In particular, $s_v=I-e_ve_v\trans(2I-A)$ is an integer matrix.
If $a=\langle e_v,x\rangle$ and $b=\langle e_v,y\rangle$, then
$\langle e_v,e_v\rangle=2$ gives
$\langle s_vx,s_vy\rangle=\langle x,y\rangle-ab-ab+2ab=\langle x,y\rangle$.
Moreover $\langle e_v,s_vx\rangle=-a$, whence $s_v^2=I$.
Thus each $s_v$ is invertible and preserves the form $C$:
$s_v\trans Cs_v=C$.

Let $\Gamma$ be the matrix group generated by the $s_v$. We call every
vector $\gamma e_v$, with $\gamma\in\Gamma$ and $v\in V$, a
\emph{real root}, and denote the set of all such vectors by $\Phi$.
Every real root is an integer vector with quadratic value one,
since $Q(\gamma e_v)=Q(e_v)=1$. In particular it is nonzero.
Also $-\gamma e_v=\gamma s_ve_v$, so $-\Phi=\Phi$.

Throughout, vector inequalities are understood coordinatewise.
Thus, \(x\ge0\) means that every coordinate of \(x\) is nonnegative.

\begin{proposition}[Real-root sign property]
\label{prop:root-sign}
Every \(\beta\in\Phi\) satisfies either $\beta\ge0$ or $\beta\le0$.
\end{proposition}

\begin{proof}
We verify the correspondence with the Coxeter-root-system framework
summarized in
\cite[Definition~1.1, Corollary~1.4, Definition~1.5, and
Proposition~1.6(ii)]{Fu12}. Fu attributes the root-basis definition
to Krammer~\cite{Krammer94} and the root-sign decomposition to
Howlett~\cite[Lectures~1 and~3]{Howlett96}.
Use the real space \(\mathbb R^V\), the simple-root set
$\Pi_0=\{e_v:v\in V\}$, and the normalized form
$(x,y)_0=\langle x,y\rangle/2$.
Then
\[
(e_v,e_v)_0=1,
\qquad
(e_u,e_v)_0=
\begin{cases}
-\frac12=-\cos(\pi/3),&u\sim v,\\
0=-\cos(\pi/2),&u\ne v,\ u\not\sim v.
\end{cases}
\]
These verify Fu's condition~(C1), with \(m_{uv}=3\) on edges
and \(m_{uv}=2\) on distinct nonedges. His positive linear cone
$\operatorname{PLC}(\Pi_0)$ consists of the vectors
$\sum_{v\in V}a_ve_v$ for which every $a_v$ is nonnegative and at
least one $a_v$ is positive.
It does not contain zero, verifying condition~(C2).

For clarity, the associated abstract Coxeter group is the presented
group
\[
W= \left\langle r_v\ (v\in V)\ \middle|\ r_v^2=1,\ (r_ur_v)^{m_{uv}}=1\ (u\ne v) \right\rangle.
\]
The assignment \(r_v\mapsto s_v\) defines a homomorphism
\(\varphi:W\to\Gamma\). Indeed, we have already shown that
\(s_v^2=I\). If \(u\) and \(v\) are nonadjacent, then
\(s_us_v=s_vs_u\), and hence \((s_us_v)^2=I\). If \(u\sim v\),
write $a=\langle e_u,x\rangle$ and $b=\langle e_v,x\rangle$.
Since \(\langle e_u,e_v\rangle=-1\), successive application gives
$s_vs_ux=x-ae_u-(a+b)e_v$,
and
$s_us_vs_ux=x-(a+b)e_u-(a+b)e_v$.
Interchanging \(u\) and \(v\) gives the same last expression. Thus
$s_us_vs_u=s_vs_us_v$ and $(s_us_v)^3=I$.
All defining relations therefore hold for the matrices.

The homomorphism \(\varphi\) is onto because its image contains every
generator of \(\Gamma\). Fu's linear reflection associated with
\(e_v\) is
$x\longmapsto x-2(x,e_v)_0e_v=s_vx$.
Hence the roots in his Definition~1.5 are exactly
$\{\varphi(w)e_v:w\in W,\ v\in V\} = \{\gamma e_v:\gamma\in\Gamma,\ v\in V\} = \Phi$.
Fu's Proposition~1.6(ii) gives a disjoint decomposition of the real
roots into positive and negative roots. Under the above identification,
the positive roots are precisely those lying in
\(\operatorname{PLC}(\Pi_0)\), hence those satisfying \(\beta\ge0\);
the negative roots are their negatives, hence those satisfying
\(\beta\le0\). This proves the proposition.
\end{proof}

An element \(\beta\in\Phi\) is called a positive root if
\(\beta\ge0\), and a negative root if \(\beta\le0\). We write
$\Phi^+=\{\beta\in\Phi:\beta\ge0\}$ and
$\Phi^-=\{\beta\in\Phi:\beta\le0\}$.
Since \(0\notin\Phi\), we have
$\Phi=\Phi^+\sqcup\Phi^-$.

For \(\alpha\in\Phi^+\), define
$s_\alpha x=x-\langle\alpha,x\rangle\alpha$.
Since \(\alpha\in\Phi^+\subseteq\Phi\), the definition of \(\Phi\)
ensures that there exist \(\gamma\in\Gamma\) and \(v\in V\) such that
$\alpha=\gamma e_v$.
Then
$\gamma s_v\gamma^{-1}x =x-\langle\gamma e_v,x\rangle\gamma e_v =x-\langle\alpha,x\rangle\alpha =s_\alpha x$.
Hence
$s_\alpha=\gamma s_v\gamma^{-1}\in\Gamma$.
For \(g\in\Gamma\), define its inversion set by
$\mathcal N(g):=\{\beta\in\Phi^+:g\beta\in\Phi^-\}$.
In particular, \(\mathcal N(s_\alpha)\) is the set of positive roots
sent to negative roots by the reflection \(s_\alpha\).

\subsection{The polynomial positivity cone}\label{subsec:cone}

For $\alpha\in\Phi^+$, define $L_\alpha=I+s_\alpha$ and
$\Pi_\alpha=\tfrac12L_\alpha=I-\tfrac12\alpha\alpha\trans C$.
Also put
$\alpha^{\perp_C}:=\{x\in\R^V:\langle\alpha,x\rangle=0\}$.
Because $\langle\alpha,\alpha\rangle=2$, the operator $\Pi_\alpha$
is the $C$-orthogonal projection onto $\alpha^{\perp_C}$ along
$\R\alpha$. Indeed, $\Pi_\alpha\alpha=0$, while
$\Pi_\alpha x=x$ for $x\in\alpha^{\perp_C}$, and hence
$\Pi_\alpha^2=\Pi_\alpha$.

The object studied in this paper is the cone
\[
\mathcal P_\alpha(A)
:=\{q\in\R[x]_{\mathrm{even}}:L_\alpha q(A)\alpha\geq0\},
\]
where \(\mathbb R[x]_{\mathrm{even}}\) denotes the set of
real polynomials of the form
$q(x)=\sum_{m=0}^{d}a_mx^{2m}$ with $a_m\in\mathbb R$.
The set
$\mathcal P_\alpha(A)$ is a convex cone: it contains zero and is
closed under addition and multiplication by nonnegative scalars.
The coordinate form of the defining condition is
\[
\bigl(L_\alpha q(A)\alpha\bigr)_r
 =2\bigl(q(A)\alpha\bigr)_r
  -\langle\alpha,q(A)\alpha\rangle\alpha_r,
 \qquad r\in V.
\]
We write \(\mathbb R_{\ge0}[x^2]\) for the cone of polynomials
$\sum_{m=0}^{d}a_mx^{2m}$ with $a_m\ge0$ and $d\in\mathbb N$.

\subsection{Nonnegative matrix polynomials}\label{subsec:matrix-polynomials}

Define polynomials in $\Z[x]$ by
\begin{equation}
 p_{-1}(x)=0,\qquad p_0(x)=1,\qquad
 p_{j+1}(x)=xp_j(x)-p_{j-1}(x)\quad(j\geq0).
 \label{eq:polynomials}
\end{equation}
Thus $p_1(x)=x$, and induction gives
$p_j(-x)=(-1)^jp_j(x)$ for $j\geq0$: multiplication by $x$
changes parity, and $p_{j-1}$ has the same parity as $p_{j+1}$. Here ``parity'' means that $p_j$ contains only even powers of $x$ when $j$ is even, and only odd powers when $j$ is odd.
For a square matrix $A$, polynomial evaluation uses $A^0=I$.
Throughout, a matrix inequality means an entrywise inequality.

A \emph{spider} is a tree with exactly one vertex of degree at least three, called its \emph{centre}. Its \emph{arms} are the paths from the centre to the leaves, and their lengths are measured in edges.

We call a connected simple graph \emph{finite Dynkin} if it is one of
$A_n$ ($n\geq1$), $D_n$ ($n\geq4$), $E_6$, $E_7$, or $E_8$.
Here $A_n$ is the path on $n$ vertices. The other graphs are spiders, which have one
vertex of degree three and three otherwise disjoint arms, whose
lengths, measured in edges, are $(1,1,n-3)$ for $D_n$, $(1,2,2)$ for
$E_6$, $(1,2,3)$ for $E_7$, and $(1,2,4)$ for $E_8$.
The term \emph{non-Dynkin} means that the graph is not in this list.

\subsubsection{Hilbert series of preprojective algebras}
\label{subsec:preprojective}

We specify the terminology needed to state the external theorem.
A finite \emph{quiver} \(\mathcal Q\) is a finite directed graph,
allowing parallel arrows and loops. Its vertex and arrow sets
are denoted by \(J\) and \(\mathcal Q_1\), and the tail and head
of an arrow \(a\) by \(t(a)\) and \(h(a)\).
Its \emph{adjacency matrix} is
$A_{\mathcal Q} := \bigl( |\{a\in\mathcal Q_1:t(a)=j,\ h(a)=i\}| \bigr)_{i,j\in J}$.
Thus parallel arrows are counted with multiplicity, and
\((A_{\mathcal Q})_{ii}\) is the number of loops at \(i\).

Connectedness refers to the underlying undirected multigraph.
A quiver is Dynkin precisely when that multigraph is one of
the simple graphs listed above; multiplicities and loops are
retained when testing this condition.
The \emph{double} \(\overline{\mathcal Q}\) is obtained by adding
a new arrow \(a^*:h(a)\to t(a)\) for each \(a\in\mathcal Q_1\),
even if an arrow in that direction is already present.
For a loop, \(a^*\) is also a distinct new loop.
Consequently,
$A_{\overline{\mathcal Q}} =A_{\mathcal Q}+A_{\mathcal Q}\trans$.

Over a field $\Bbbk$, the \emph{path algebra}
$\Bbbk\overline{\mathcal Q}$ has as a basis all directed paths,
including a length-zero path $\varepsilon_i$ at each vertex $i$.
Products are concatenations, with the right factor traversed first,
and are zero when the endpoints do not match. Multiplication extends
bilinearly. The elements $\varepsilon_i$ are orthogonal idempotents,
and $\sum_{i\in J}\varepsilon_i$ is the identity. In particular,
$R=\bigoplus_{i\in J}\Bbbk\varepsilon_i$
is the algebra of $\Bbbk$-valued functions on $J$, with pointwise
multiplication. The \emph{preprojective algebra} is
$\Pi_{\mathcal Q}(\Bbbk)=\Bbbk\overline{\mathcal Q}/(\rho)$, where
$\rho=\sum_{a\in\mathcal Q_1}(aa^*-a^*a)$ and $(\rho)$ is the
two-sided ideal of finite sums of terms
$u\rho v$. Equivalently, its defining relations are
$\rho_i=\sum_{h(a)=i}aa^*-\sum_{t(a)=i}a^*a=0$ for $i\in J$.
Indeed, $\rho_i=\varepsilon_i\rho\varepsilon_i$ and
$\rho=\sum_i\rho_i$, so these elements generate the same ideal.

Path length gives the path algebra a grading
$\Bbbk\overline{\mathcal Q}=\bigoplus_{d\geq0}
(\Bbbk\overline{\mathcal Q})[d]$, with products adding degrees.
Each $\rho_i$ has degree two. To see explicitly that the quotient
inherits this grading, write an element of the relation ideal as
a finite sum $\sum_\ell c_\ell u_\ell\rho_{i_\ell}v_\ell$, with
$u_\ell,v_\ell$ paths. Its degree-$d$ component is the sum over
the indices satisfying $|u_\ell|+2+|v_\ell|=d$, and hence still
belongs to the relation ideal. Consequently
$\Pi_{\mathcal Q}(\Bbbk)[d] \cong (\Bbbk\overline{\mathcal Q})[d] \big/\bigl((\rho)\cap(\Bbbk\overline{\mathcal Q})[d]\bigr)$.
These spaces are finite dimensional because the quiver is finite
and only finitely many paths have a given length. The degree-zero
part is $R$, since the relation ideal has no nonzero component
of degree less than two.

Left and right multiplication by $R$ commute, making each graded
component an $R$-bimodule. The orthogonal idempotents give its
direct-sum decomposition into the spaces
$\varepsilon_i\Pi_{\mathcal Q}(\Bbbk)[d]\varepsilon_j$.
Before quotienting, this $(i,j)$-space consists of paths from $j$
to $i$. The \emph{matrix Hilbert series} recording the dimensions of
these spaces is $h_{\Pi_{\mathcal Q}}(t)=\sum_{d\geq0}H_dt^d$, where
$(H_d)_{ij}=\dim_{\Bbbk}\bigl(\varepsilon_i\Pi_{\mathcal Q}(\Bbbk)[d]\varepsilon_j\bigr)$.
In particular, $H_0=I$ and every entry of every $H_d$ is a
nonnegative integer, independently of the characteristic of $\Bbbk$.

All power series in this subsection are formal. Multiplication of
matrix series is defined by the finite convolution
$[t^d]\left(\sum_{a\geq0}M_at^a\right) \left(\sum_{b\geq0}N_bt^b\right) =\sum_{a=0}^d M_aN_{d-a}$.
For a fixed matrix $D$, the expression $(I-Dt+It^2)^{-1}$ is
well-defined: writing $X=Dt-It^2$, the series
$\sum_{a\geq0}X^a$ is a two-sided inverse of $I-X$.
For each coefficient only finitely many summands contribute,
so the finite geometric-sum identity proves both inverse identities.

The following theorem gives the Hilbert series in the form required below \cite[Theorem~3.4.2]{EE07}.

\begin{theorem}[Etingof--Eu]\label{thm:etingof-eu}
Let \(\mathcal Q\) be a connected non-Dynkin quiver.
Then the matrix Hilbert series of its preprojective algebra
over an arbitrary field is
$h_{\Pi_{\mathcal Q}}(t) = \bigl(I-A_{\overline{\mathcal Q}}t+t^2I\bigr)^{-1}$.
\end{theorem}

The graph-theoretic consequence needed here is a separate statement.

\begin{proposition}
\label{prop:matrix-positive}
Let $A$ be the adjacency matrix of a finite, nonempty, connected
simple graph that is not finite Dynkin. Then
$p_j(A)\geq0$ entrywise for every $j\geq0$.
\end{proposition}

\begin{proof}
Orient each edge of the graph exactly once to obtain a finite quiver
$\mathcal Q$, and take $\Bbbk=\C$. Its underlying undirected graph
is the given graph, so the connectedness and non-Dynkin hypotheses
of the quoted theorem hold. The adjacency matrix of
$\overline{\mathcal Q}$ is exactly $A$: each original edge produces
one arrow in each direction, and hence one entry in each of the
two corresponding ordered matrix positions.

Let $H_d$ be the dimension matrices just defined. The theorem gives
$(I-At+It^2)h_{\Pi_{\mathcal Q}}(t)=I$.
The constant and linear coefficients give $H_0=I$ and $H_1=A$.
For every $d\geq2$, coefficient comparison gives
$H_d-AH_{d-1}+H_{d-2}=0$.
Polynomial evaluation in \eqref{eq:polynomials} gives the same
initial values and recurrence for $p_d(A)$. Induction therefore
yields $H_d=p_d(A)$ for all $d\geq0$. More explicitly,
$(p_d(A))_{ij} =\dim_{\C}\bigl( \varepsilon_i\Pi_{\mathcal Q}(\C)[d]\varepsilon_j\bigr) \in\Z_{\geq0}$.
This proves the proposition; also $p_{-1}(A)=0$ by definition.
\end{proof}
\begin{lemma}\label{lem:nonnegative-expansion}
For every \(m\ge0\) there are integers \(c_{m,j}\ge0\), \(0\le j\le m\),
such that
\[
x^{2m}=\sum_{j=0}^m c_{m,j}p_{2j}(x).
\]
\end{lemma}

\begin{proof}
The case \(m=0\) is \(1=p_0(x)\).
The defining recurrence gives
$x^2p_0(x)=p_2(x)+p_0(x)$
and, for \(j\ge1\),
$x^2p_{2j}(x)=p_{2j+2}(x)+2p_{2j}(x)+p_{2j-2}(x)$.
Indeed, apply \(xp_a=p_{a+1}+p_{a-1}\) twice; the separate formula
at \(j=0\) avoids any use of \(p_{-2}\).
Multiplying such an expansion by \(x^2\)
and using these formulas produces only nonnegative integer
coefficients, with indices from \(0\) to \(2m+2\).
This completes the induction.
\end{proof}

The recurrence also gives
$x^{2m}\in\operatorname{cone}\{p_0,p_2,\ldots,p_{2m}\}$ for every
$m\ge0$. Hence every polynomial in $\mathbb R_{\ge0}[x^2]$ belongs to
the cone generated by $p_0,p_2,p_4,\ldots$.

\section{A positivity criterion for positive real roots}
\label{sec:root-positivity}
Throughout this section, \(T=(V,E)\) denotes a graph with adjacency
matrix \(A\). When \(T\) is bipartite, we write
\(V=X\sqcup Y\) for a fixed bipartition.

For later use, the recurrence defining \(p_j\) gives
$p_a(2\cos\theta)=\sin((a+1)\theta)/\sin\theta$ whenever $a\ge-1$ and
$0<\theta<\pi$.
Indeed, the right-hand side has the prescribed initial values and
satisfies the same recurrence.

\begin{lemma}\label{lem:two-colour}
Let \(T\) be a finite connected bipartite simple graph with bipartition
\(V=X\sqcup Y\). Write the coordinates in the order \(X,Y\), so that
the adjacency matrix of $T$ has the form
\[
A=\begin{pmatrix}0&M\\M\trans&0\end{pmatrix}
\]
Let \(R_X\) and \(R_Y\) be the products of the simple reflections in
the vertices of \(X\) and \(Y\), respectively. Then
\[
R_X=\begin{pmatrix}-I&M\\0&I\end{pmatrix},
\qquad
R_Y=\begin{pmatrix}I&0\\M\trans&-I\end{pmatrix}.
\]
If \(\Omega=R_XR_Y\), then \(\Omega^{-1}=R_YR_X\) and
$\Omega+\Omega^{-1}=A^2-2I$.
Moreover, for every \(j\ge1\),
$\Omega^j+\Omega^{-j} = p_{2j}(A)-p_{2j-2}(A)$.
Consequently, for every vector \(\alpha\) and every \(m\ge0\),
\begin{equation}\label{eq:alpha-expansion}
p_{2m}(A)\alpha
=
\alpha+\sum_{j=1}^m
(\Omega^j+\Omega^{-j})\alpha.
\end{equation}
\end{lemma}

\begin{proof}
Reflections at vertices of the same colour commute, because they
alter different coordinates and each altered coordinate depends only
on the opposite colour. Writing vectors in the order \(X,Y\), we have
\[
R_X
=
\begin{pmatrix}
-I&M\\
0&I
\end{pmatrix},
\qquad
R_Y
=
\begin{pmatrix}
I&0\\
M\trans&-I
\end{pmatrix}.
\]
Indeed,
\[
R_X
\begin{pmatrix}x\\y\end{pmatrix}
=
\begin{pmatrix}My-x\\y\end{pmatrix},
\qquad
R_Y
\begin{pmatrix}x\\y\end{pmatrix}
=
\begin{pmatrix}x\\M\trans x-y\end{pmatrix}.
\]
The preceding formulas give
\[
\begin{aligned}
\Omega=R_XR_Y
&=
\begin{pmatrix}
MM\trans-I&-M\\
M\trans&-I
\end{pmatrix},\\
\Omega^{-1}=R_YR_X
&=
\begin{pmatrix}
-I&M\\
-M\trans&M\trans M-I
\end{pmatrix}.
\end{aligned}
\]
The second identity uses \(R_X^2=R_Y^2=I\). Adding the two matrices
gives $\Omega+\Omega^{-1}=A^2-2I$, which proves the first asserted
identity.
Set \(q_0(x)=2\) and
\(q_j(x)=p_{2j}(x)-p_{2j-2}(x)\) for \(j\ge1\).
By the sine representation of \(p_a\) above,
$q_j(2\cos\theta)=2\cos(2j\theta)$ for every $j\ge0$.
The cosine addition formula therefore gives
$q_{j+1}(x)=(x^2-2)q_j(x)-q_{j-1}(x)$ for $j\ge1$.
Now set $S_j=\Omega^j+\Omega^{-j}$ for $j\ge0$.
Using \(\Omega+\Omega^{-1}=A^2-2I\), for \(j\ge1\) we obtain
\[
\begin{aligned}
(A^2-2I)S_j
&=(\Omega+\Omega^{-1})(\Omega^j+\Omega^{-j})\\
&=\Omega^{j+1}+\Omega^{j-1}
  +\Omega^{-(j-1)}+\Omega^{-(j+1)}\\
&=S_{j+1}+S_{j-1}.
\end{aligned}
\]
Hence
$S_{j+1}=(A^2-2I)S_j-S_{j-1}$.
On the other hand, evaluating the recurrence for \(q_j\) at \(A\)
gives
$q_{j+1}(A)=(A^2-2I)q_j(A)-q_{j-1}(A)$.
The initial values $S_0=2I=q_0(A)$ and
$S_1=\Omega+\Omega^{-1}=A^2-2I=q_1(A)$ also agree.
Therefore, by induction,
$\Omega^j+\Omega^{-j}=q_j(A)=p_{2j}(A)-p_{2j-2}(A)$ for every
$j\ge1$.
This proves the second asserted identity. Telescoping it gives
\eqref{eq:alpha-expansion}.
\end{proof}

\begin{lemma}\label{lem:positive-eigenvector}
Let \(A\) be the adjacency matrix of a finite connected non-Dynkin
simple graph $T$. The largest eigenvalue \(\lambda\) of \(A\) satisfies
\(\lambda\ge2\) and has an eigenvector \(h>0\). For every \(j\ge1\),
$p_{2j}(\lambda)-p_{2j-2}(\lambda)\ge2$.
\end{lemma}

\begin{proof}
Since $A$ is nonnegative and $T$ is connected, the
Perron--Frobenius theorem gives an eigenvector $h>0$ satisfying
$Ah=\lambda h$, where $\lambda=\rho(A)\ge0$.
Smith's classification, in the form recorded in \cite[Section~3.1.1]{BH12}, states that the connected graphs with adjacency spectral radius less than $2$ are precisely the finite Dynkin diagrams $A_n,D_n,E_6,E_7,E_8$. Since $T$ lies outside this list, $\lambda\ge2$.

If $\lambda=2$, the defining recurrence gives
$p_a(2)=a+1$ for $a\ge-1$, and hence
$p_{2j}(2)-p_{2j-2}(2)=2$.
If $\lambda>2$, write $\lambda=2\cosh\eta$ with $\eta>0$.
The same recurrence and initial values give
$p_a(\lambda)=\sinh((a+1)\eta)/\sinh\eta$ for every $a\ge-1$.
Consequently,
\[
\begin{aligned}
p_{2j}(\lambda)-p_{2j-2}(\lambda)
&=\frac{\sinh((2j+1)\eta)-\sinh((2j-1)\eta)}
        {\sinh\eta}\\
&=2\cosh(2j\eta)>2.
\end{aligned}
\]
This proves the asserted bound.
\end{proof}
We now turn to the proof of Theorem~~\ref{thm:root-positivity}. For convenience, we restate
the theorem before giving its proof.
\begin{theorem2}[Root-positivity criterion]
Let \(A\) be the adjacency matrix of a finite connected bipartite
non-Dynkin simple graph $T$, and let \(\alpha\in\Phi^+\). Put
\(\Omega=R_XR_Y\), with respect to a fixed bipartition
\(V=X\sqcup Y\). Suppose that at least one of the following conditions
holds:
\begin{enumerate}
\item \(\Omega^j\alpha\in\Phi^-\) for some \(j\in\mathbb Z\);
\item every \(\beta\in\mathcal N(s_\alpha)\) satisfies
      \(0\le\beta\le\alpha\).
\end{enumerate}
Then $L_\alpha p_{2m}(A)\alpha\ge0$ for every $m\ge0$.
Consequently,
$\operatorname{cone}\{p_0,p_2,p_4,\ldots\}\subseteq\mathcal P_\alpha(A)$
and $\mathbb R_{\ge0}[x^2]\subseteq\mathcal P_\alpha(A)$.
\end{theorem2}

\begin{proof}
The two-colour identities in Lemma~\ref{lem:two-colour} and the
spectral estimate in Lemma~\ref{lem:positive-eigenvector} will be used
throughout.

\textbf{Case 1.}
Suppose that \(\Omega^j\alpha\in\Phi^-\) for some \(j\in\mathbb Z\).
Since \(\alpha\in\Phi^+\), we have \(j\ne0\). Put \(N=2|j|\), and
define the sequence of colours \(U_1,\ldots,U_N\) by
\[
U_a=
\begin{cases}
Y,& j>0\text{ and }a\text{ is odd},\\
X,& j>0\text{ and }a\text{ is even},\\
X,& j<0\text{ and }a\text{ is odd},\\
Y,& j<0\text{ and }a\text{ is even}.
\end{cases}
\]
Thus
$\Omega^j=R_{U_N}\cdots R_{U_1}$. Define $\gamma_0=\alpha$ and
$\gamma_a=R_{U_a}\gamma_{a-1}$ for $1\le a\le N$. Let
$t=\min\{a\in\{1,\ldots,N\}:\gamma_a\in\Phi^-\}$ and set
$\beta=\gamma_{t-1}$.
Then \(\beta\in\Phi^+\) and
\(R_{U_t}\beta=\gamma_t\in\Phi^-\).

The colour reflection \(R_{U_t}\) changes only the coordinates in
\(U_t\). Therefore, for $x\notin U_t$,
$0\le\beta_x=(R_{U_t}\beta)_x\le0$,
so \(\operatorname{supp}(\beta)\subseteq U_t\). Since \(U_t\) is an
independent set,
$Q(\beta)=\sum_{x\in V(T)}\beta_x^2-\sum_{xy\in E(T)}\beta_x\beta_y
=\sum_{x\in U_t}\beta_x^2$.
All simple reflections have integer entries and \(\alpha\) is an
integral real root, so \(\beta\) is integral. Since \(Q(\beta)=1\)
and \(\beta\ge0\), there is a unique \(v\in U_t\) such that
$\beta=e_v$.

Since each \(R_{U_a}\) is an involution, reversing the preceding
reflections recovers \(\gamma_0=\alpha\) from
\(\gamma_{t-1}=e_v\).
We claim that
$\gamma_{t-1-a}=\bigl(p_a(A)+p_{a-1}(A)\bigr)e_v$ for
$0\le a\le t-1$.
For \(a=0\), this is \(\gamma_{t-1}=e_v\). Suppose it holds for
\(0\le a\le t-2\). The next reflection is \(R_{U_{t-1-a}}\). Since the colours
alternate, \(U_{t-1-a}\) is the colour class in which
\(p_{a-1}(A)e_v\) is supported, whereas \(p_a(A)e_v\) is supported
in the opposite colour class. Hence \(R_{U_{t-1-a}}\) leaves
\(p_a(A)e_v\) unchanged and acts only on the coordinates in the
colour class of \(p_{a-1}(A)e_v\). Therefore,
\[
\begin{aligned}
\gamma_{t-2-a}
&=R_{U_{t-1-a}}\gamma_{t-1-a}\\
&=p_a(A)e_v+
  \bigl(Ap_a(A)-p_{a-1}(A)\bigr)e_v\\
&=\bigl(p_{a+1}(A)+p_a(A)\bigr)e_v.
\end{aligned}
\]
Set $H:=p_{t-1}(A)+p_{t-2}(A)$. Taking \(a=t-1\) we have
$\gamma_0=He_v=\alpha$. In particular, $H_{rv}=\alpha_r$ for every
$r\in V$.

We next verify
\[
(x-2)\bigl(p_{t-1}(x)+p_{t-2}(x)\bigr)^2
=
p_{2t-1}(x)-p_{2t-3}(x)-2.
\]
Let $x=2\cos\theta$, where $0<\theta<\pi$.
By the sine representation of \(p_a\) above and the sine addition formula,
\[
\begin{aligned}
p_{t-1}(x)+p_{t-2}(x)
&=\frac{\sin(t\theta)+\sin((t-1)\theta)}
        {\sin\theta}\\
&=\frac{
  2\sin((2t-1)\theta/2)\cos(\theta/2)}
  {2\sin(\theta/2)\cos(\theta/2)}\\
&=\frac{\sin((2t-1)\theta/2)}{\sin(\theta/2)}.
\end{aligned}
\]
Using $x-2=-4\sin^2(\theta/2)$, we get
\[
\begin{aligned}
(x-2)\bigl(p_{t-1}(x)+p_{t-2}(x)\bigr)^2
&=-4\sin^2((2t-1)\theta/2)\\
&=2\cos((2t-1)\theta)-2.
\end{aligned}
\]
On the other hand, the sine subtraction formula gives
\[
\begin{aligned}
p_{2t-1}(x)-p_{2t-3}(x)
&=\frac{\sin(2t\theta)-\sin((2t-2)\theta)}
        {\sin\theta}\\
&=\frac{2\cos((2t-1)\theta)\sin\theta}{\sin\theta}\\
&=2\cos((2t-1)\theta).
\end{aligned}
\]
Thus the two sides of the claimed polynomial identity agree for every
$x\in(-2,2)$. Their difference is a polynomial with infinitely
many zeros, so it is the zero polynomial.
The convention $p_{-1}=0$ includes $t=1$, when the
identity reads $x-2=x-2$.
Consequently, substitution of any square matrix for $x$ is valid;
in particular,
$(A-2I)H^2 =p_{2t-1}(A)-p_{2t-3}(A)-2I$.

Fix \(m\ge0\). The matrices \(p_{2m}(A),H,A\) are symmetric and
commute. Using \(\alpha=He_v\) and the resulting matrix identity,
\[
\begin{aligned}
\alpha\trans(A-2I)p_{2m}(A)\alpha
&=e_v\trans H\trans(A-2I)p_{2m}(A)H e_v\\
&=e_v\trans p_{2m}(A)(A-2I)H^2e_v\\
&=e_v\trans
p_{2m}(A)
\bigl(p_{2t-1}(A)-p_{2t-3}(A)-2I\bigr)e_v\\
&=-2\bigl(p_{2m}(A)\bigr)_{vv}.
\end{aligned}
\]
Indeed, the polynomial
$p_{2m}(x)\bigl(p_{2t-1}(x)-p_{2t-3}(x)\bigr)$
is odd, and the diagonal entries of every odd power of a bipartite
adjacency matrix vanish. Therefore
$\langle\alpha,p_{2m}(A)\alpha\rangle =2\bigl(p_{2m}(A)\bigr)_{vv}$.
Using the coordinate formula for \(L_\alpha\), \(\alpha=He_v\), and the fact that
\(p_{2m}(A)H=Hp_{2m}(A)\), we obtain that for any $r\in V$
\[
\begin{aligned}
(L_\alpha p_{2m}(A)\alpha)_r
&=2\bigl(p_{2m}(A)\alpha\bigr)_r
 -2\bigl(p_{2m}(A)\bigr)_{vv}\alpha_r\\
&=2\sum_{u\in V}H_{ru}\bigl(p_{2m}(A)\bigr)_{uv}
 -2\bigl(p_{2m}(A)\bigr)_{vv}H_{rv}\\
&=2\sum_{\substack{u\in V\\u\ne v}}
H_{ru}\bigl(p_{2m}(A)\bigr)_{uv}\ge0.
\end{aligned}
\]
The last inequality follows from
\(p_j(A)\ge0\) for \(j\ge0\), together with
\(p_{-1}(A)=0\).
This proves the theorem under condition~1.

\textbf{Case 2.}
Suppose now that every \(\Omega^j\alpha\), \(j\in\mathbb Z\), is positive
and that condition~2 holds. If
\(\beta\in\Phi^+\setminus\mathcal N(s_\alpha)\), then
\(s_\alpha\beta\in\Phi^+\). For every $r\in V$, its \(r\)-coordinate gives
$\beta_r\ge\langle\alpha,\beta\rangle\alpha_r$.

Fix \(j\ge1\), and put $x=\Omega^j\alpha$ and
$y=\Omega^{-j}\alpha$.
Since \(\Omega\in\Gamma\) preserves the bilinear form and the form is
symmetric, we have
$\langle\alpha,\Omega^j\alpha\rangle =\langle\Omega^{-j}\alpha,\alpha\rangle =\langle\alpha,\Omega^{-j}\alpha\rangle$.
Hence set
$b:=\langle\alpha,x\rangle=\langle\alpha,y\rangle$.
By the definition of $L_\alpha$,
$\bigl(L_\alpha(x+y)\bigr)_r = 2\bigl(x_r+y_r-b\alpha_r\bigr)$.
If \(b\le0\), this is nonnegative because \(x,y,\alpha\ge0\).
If \(b>0\) and at least one of \(x,y\) is not in
\(\mathcal N(s_\alpha)\), then the coordinate bound above shows
that the same expression is again nonnegative.

It remains to consider the case \(x,y\in\mathcal N(s_\alpha)\). By
condition~2,
$0\le x,y\le\alpha$.
If \(x=\alpha\), then \(\Omega^j\alpha=\alpha\). Applying
\(\Omega^{-j}\) to both sides gives
$y=\Omega^{-j}\alpha=\alpha$. In this case
\(L_\alpha(x+y)=2L_\alpha\alpha=0\). Otherwise, neither \(x\) nor \(y\) is equal to \(\alpha\). Since
\(0\le x,y\le\alpha\), each of the vectors \(x\) and \(y\) has at least
one coordinate strictly smaller than the corresponding coordinate of
\(\alpha\). Thus \(2\alpha-(x+y)\) is nonzero nonnegative vector, and hence
$h\trans(x+y)<2h\trans\alpha$.
Here \(h>0\) is the Perron vector from
Lemma~\ref{lem:positive-eigenvector}. On the other hand,
Lemma~\ref{lem:two-colour} and \(Ah=\lambda h\) give
$h\trans(x+y) = \bigl(p_{2j}(\lambda)-p_{2j-2}(\lambda)\bigr)h\trans\alpha \ge2h\trans\alpha$,
a contradiction. Thus
$L_\alpha(\Omega^j+\Omega^{-j})\alpha\ge0$ for every $j\ge1$.
Finally, \(L_\alpha\alpha=0\), and
\eqref{eq:alpha-expansion} yields
$L_\alpha p_{2m}(A)\alpha = \sum_{j=1}^m L_\alpha(\Omega^j+\Omega^{-j})\alpha\ge0$.
This proves the first conclusion of the theorem. The two cone
inclusions follow from linearity, the definition of
\(\mathcal P_\alpha(A)\), and the monomial-cone inclusion proved above.
\end{proof}

\section{Indicator roots in bipartite graphs and walk inequalities}
\label{sec:trees}

Throughout this section, \(T=(V,E)\) denotes a graph with adjacency
matrix \(A\). For
\(W\subseteq V\), let \(\one_W\in\mathbb R^V\) denote the indicator
vector of \(W\), and write \(\one:=\one_V\). For \(W\subseteq V\), we write \(T[W]\) for the subgraph induced by
\(W\).

\subsection{Indicator roots and their inversion sets}
\label{subsec:indicator-roots}

\begin{lemma}\label{lem:inversion-bound}
If \(g\in\Gamma\) has an expression as a product of \(L\) simple
reflections, then
$|\mathcal N(g)|\le L$.
\end{lemma}

\begin{proof}
We first determine the inversion set of a single simple reflection.
Let \(\beta\in\Phi^+\) and suppose that
\(s_v\beta\in\Phi^-\). For every \(u\ne v\), the reflection \(s_v\)
does not change the \(u\)-coordinate, so
$0\le\beta_u=(s_v\beta)_u\le0$.
Thus \(\beta_u=0\) for every \(u\ne v\), and hence
$\beta=t e_v$
for some \(t>0\). Since \(Q(\beta)=1\) and \(Q(e_v)=1\), we obtain
$1=Q(\beta)=t^2Q(e_v)=t^2$.
Therefore \(t=1\), so \(\beta=e_v\). Conversely,
$s_ve_v=e_v-\langle e_v,e_v\rangle e_v=-e_v\in\Phi^-$.
Hence
$\mathcal N(s_v)=\{e_v\}$.

Now suppose that $g=s_{a_L}\cdots s_{a_1}$. Set $g_0:=I$ and
$g_j:=s_{a_j}\cdots s_{a_1}$ for $1\le j\le L$.
For \(1\le j\le L\), define
\[
\mathcal N_j(g):=
\left\{
\beta\in\mathcal N(g):
\begin{array}{l}
g_i\beta\in\Phi^+\quad(0\le i<j),\\
g_j\beta\in\Phi^-
\end{array}
\right\}.
\]
Every \(\beta\in\mathcal N(g)\) starts as a positive root and ends as
a negative root. Therefore, by the root sign decomposition, there is a
unique first index at which it becomes negative. Consequently,
$\mathcal N(g)=\bigsqcup_{j=1}^{L}\mathcal N_j(g)$.

Let \(\beta\in\mathcal N_j(g)\). Then $g_{j-1}\beta\in\Phi^+$ and
$s_{a_j}g_{j-1}\beta=g_j\beta\in\Phi^-$.
The result for a single simple reflection implies that
$g_{j-1}\beta=e_{a_j}$.
Since
$g_{j-1}^{-1}=s_{a_1}\cdots s_{a_{j-1}}$,
we have
$\beta =g_{j-1}^{-1}e_{a_j} =s_{a_1}\cdots s_{a_{j-1}}e_{a_j}$.
Thus \(\mathcal N_j(g)\) contains at most one root, and therefore
$|\mathcal N(g)| =\sum_{j=1}^{L}|\mathcal N_j(g)| \le L$.
If \(L=0\), then \(g=I\) and
\(\mathcal N(I)=\varnothing\), so the same conclusion holds.
\end{proof}

\begin{lemma}\label{lem:indicator-inversion}
Let \(T\) be a finite connected bipartite simple graph, and let \(U\subseteq V(T)\) be nonempty such that $T[U]$ is a tree. For every
\(e\in E(T[U])\), let \(U_{e,1}\) and \(U_{e,2}\) be the vertex sets
of the two components of \(T[U]-e\). Then
$\mathcal N(s_{\one_U})$ consists precisely of $\one_U$ and, for every
$e\in E(T[U])$, the two indicator roots $\one_{U_{e,1}}$ and
$\one_{U_{e,2}}$; all these roots are distinct. In particular, every
$\beta\in\mathcal N(s_{\one_U})$ satisfies $0\le\beta\le\one_U$, and
$|\mathcal N(s_{\one_U})|=2|U|-1$.
\end{lemma}

\begin{proof}
Put \(q:=|U|\). We first show that the indicator vector of every
nonempty connected subset of \(U\) is a positive real root.

Root \(T[U]\) at \(u_1\), and order its vertices
\(u_1,\ldots,u_q\) so that every parent precedes its children. For
$1\le j\le q$, set $U_j:=\{u_1,\ldots,u_j\}$.
For \(2\le j\le q\), let \(p_U(u_j)\) be the parent of \(u_j\) in
\(T[U]\). Since \(T[U]\) is a tree and the parent precedes its
children,
$N_T(u_j)\cap U_{j-1}=\{p_U(u_j)\}$.
The coordinate formula for a simple reflection gives
\[
\bigl(s_{u_j}\one_{U_{j-1}}\bigr)_x
=
\begin{cases}
\displaystyle
\sum_{y\sim u_j}(\one_{U_{j-1}})_y
-(\one_{U_{j-1}})_{u_j}
=1,
&x=u_j,\\[6pt]
(\one_{U_{j-1}})_x,
&x\ne u_j.
\end{cases}
\]
Indeed, \(u_j\notin U_{j-1}\), while its unique neighbour in
\(U_{j-1}\) is \(p_U(u_j)\). Hence $e_{u_1}=\one_{U_1}$ and
$s_{u_j}\one_{U_{j-1}}=\one_{U_j}$ for $2\le j\le q$.
It follows inductively that
$\one_U =s_{u_q}\cdots s_{u_2}e_{u_1}\in\Phi^+$.
When \(q=1\), the product is understood to be empty and equal to \(I\). The same construction, applied to any nonempty
\(W\subseteq U\) for which \(T[W]\) is connected, shows that
\(\one_W\in\Phi^+\).

In particular,
$Q(\one_U)=|U|-|E(T[U])|=|U|-(|U|-1)=1$. Let
$w:=s_{u_q}\cdots s_{u_2}$. Then $we_{u_1}=\one_U$ and
$w^{-1}=s_{u_2}\cdots s_{u_q}$.
Since \(w\) preserves the bilinear form, for every vector \(x\) we have
\[
\begin{aligned}
ws_{u_1}w^{-1}x
&=x-\langle e_{u_1},w^{-1}x\rangle\,we_{u_1}\\
&=x-\langle we_{u_1},x\rangle\,we_{u_1}\\
&=x-\langle\one_U,x\rangle\one_U\\
&=s_{\one_U}x.
\end{aligned}
\]
Therefore, $s_{\one_U}=ws_{u_1}w^{-1}$ has an expression involving
$(q-1)+1+(q-1)=2q-1$ simple reflections. Lemma~\ref{lem:inversion-bound}
consequently gives $|\mathcal N(s_{\one_U})|\le2q-1$.

It remains to exhibit \(2q-1\) distinct roots in
\(\mathcal N(s_{\one_U})\). Fix \(e\in E(T[U])\), and let
\(W=U_{e,1}\) and \(U\setminus W=U_{e,2}\). Both \(W\) and
\(U\setminus W\) are nonempty connected subsets of \(U\), so the
construction above shows that $\one_W$ and $\one_{U\setminus W}$
belong to $\Phi^+$.
The only edge of \(T[U]\) joining \(W\) to \(U\setminus W\) is \(e\).
Since \(T[W]\) is a tree,
$|E(T[W])|=|W|-1$.
Thus
\[
\begin{aligned}
\langle\one_U,\one_W\rangle
&=2|W|-2|E(T[W])|-1\\
&=2|W|-2(|W|-1)-1\\
&=1.
\end{aligned}
\]
Consequently,
$s_{\one_U}\one_W=\one_W-\langle\one_U,\one_W\rangle\one_U
=-\one_{U\setminus W}\in\Phi^-$. Therefore
$\one_W\in\mathcal N(s_{\one_U})$.
Interchanging \(W\) and \(U\setminus W\) gives
$\one_{U\setminus W}\in\mathcal N(s_{\one_U})$.

Moreover, $\langle\one_U,\one_U\rangle=2Q(\one_U)=2$, so
$s_{\one_U}\one_U=-\one_U\in\Phi^-$ and
$\one_U\in\mathcal N(s_{\one_U})$. Thus $\one_U$ and the two branch
indicators associated with every edge belong to
$\mathcal N(s_{\one_U})$. These roots are pairwise distinct, and their
number is $1+2|E(T[U])|=1+2(q-1)=2q-1$. Together with the upper bound,
this proves the asserted description of $\mathcal N(s_{\one_U})$ and
the identity $|\mathcal N(s_{\one_U})|=2|U|-1$.
\end{proof}
\subsection{Polynomial positivity for indicator roots}
\label{subsec:indicator-positivity}

\begin{proposition}\label{prop:indicator-positivity}
Let \(T\) be a finite non-Dynkin connected bipartite simple graph, let $A$ be the adjacency matrix of $T$, and let \(U\subseteq V(T)\) be
nonempty such that $T[U]$ is a tree. Then both
$\operatorname{cone}\{p_0,p_2,p_4,\ldots\}$ and
$\mathbb R_{\ge0}[x^2]$ are contained in $\mathcal P_{\one_U}(A)$.
\end{proposition}

\begin{proof}
By Lemma~\ref{lem:indicator-inversion}, every
$\beta\in\mathcal N(s_{\one_U})$ satisfies $0\le\beta\le\one_U$.
Therefore condition~2 of Theorem~\ref{thm:root-positivity} applies with
\(\alpha=\one_U\), and the result follows.
\end{proof}
\subsection{The walk interpretation}
\label{subsec:walks}

For \(W\subseteq V\), \(r\in V\), and \(k\ge0\), define
$f_k(T,W,r):=e_r\trans A^k\one_W$ and
$w_k(T,W):=\one_W\trans A^k\one_W$.
Thus \(f_k(T,W,r)\) counts walks of length \(k\) starting at \(r\)
and ending in \(W\), whereas \(w_k(T,W)\) counts walks whose initial
and terminal vertices both lie in \(W\).

When \(W=V\), abbreviate
$f_k(T,r):=f_k(T,V,r)=e_r\trans A^k\one$ and
$w_k(T):=w_k(T,V)=\one\trans A^k\one$.
When \(T\) is fixed, the argument \(T\) may be omitted.

For \(W\subseteq V\), define
$\chi_{W,r}:=2e_r-(\one_W)_rC\one_W$ and
$S_k(T,W,r):=\chi_{W,r}\trans A^k\one_W$.
Since \(C=2I-A\),
$S_k(T,W,r) = 2f_k(T,W,r) -(\one_W)_r \bigl(2w_k(T,W)-w_{k+1}(T,W)\bigr)$.

For \(W=V\), write
$\chi_r:=\chi_{V,r}=(A-2I)\one+2e_r$ and
$S_k(T,r):=S_k(T,V,r)$.
Thus
\begin{equation}\label{eq:rooted-def}
S_k(T,r)
=\chi_r\trans A^k\one
=w_{k+1}(T)-2w_k(T)+2f_k(T,r).
\end{equation}

Since \(C=2I-A\) and \(\alpha=\one\),
$e_r\trans L_\one=\chi_r\trans$.
Consequently, for every polynomial \(q\),
$\bigl(L_\one q(A)\one\bigr)_r = \chi_r\trans q(A)\one$.
Summing \eqref{eq:rooted-def} over all roots gives
\begin{equation}\label{eq:sum-roots}
\sum_{r\in V}S_k(T,r)
=
n w_{k+1}(T)-2(n-1)w_k(T).
\end{equation}

\begin{proposition}\label{prop:nondynkin-rooted}
Let \(T\) be a finite non-Dynkin tree. Then
$S_{2m}(T,r)\ge0$ for every \(r\in V(T)\) and every \(m\ge0\).
\end{proposition}

\begin{proof}
Apply Proposition~\ref{prop:indicator-positivity} with \(U=V(T)\).
Since \(x^{2m}\in\mathbb R_{\ge0}[x^2]\), we have
\(x^{2m}\in\mathcal P_\one(A)\). The preceding correspondence then gives
$S_{2m}(T,r)=\chi_r\trans A^{2m}\one=(L_\one A^{2m}\one)_r$, which
is nonnegative.
\end{proof}

\begin{proposition}\label{prop:indicator-rooted}
Let \(T\) be a finite connected bipartite non-Dynkin simple graph
with adjacency matrix \(A\), and let
\(\varnothing\ne U\subseteq V(T)\) be such that \(T[U]\) is a tree.
Then $S_{2m}(T,U,r)\ge0$ for every \(r\in V(T)\) and every \(m\ge0\).
\end{proposition}

\begin{proof}
Since \(x^{2m}\in\mathbb R_{\ge0}[x^2]\), Proposition
\ref{prop:indicator-positivity} gives
$x^{2m}\in\mathcal P_{\one_U}(A)$ and hence
$L_{\one_U}A^{2m}\one_U\ge0$. Moreover,
$e_r\trans L_{\one_U}=2e_r\trans-(\one_U)_r\one_U\trans C
=\chi_{U,r}\trans$.
Therefore
\[
\begin{aligned}
S_{2m}(T,U,r)
&=\chi_{U,r}\trans A^{2m}\one_U\\
&=e_r\trans L_{\one_U}A^{2m}\one_U\\
&=\bigl(L_{\one_U}A^{2m}\one_U\bigr)_r
\ge0.
\end{aligned}
\]
\end{proof}

\section{Finite Dynkin trees}\label{sec:finite}

We now treat the trees excluded by
Proposition~\ref{prop:matrix-positive}. The argument uses only
walk-counting generating functions and the explicit Dynkin list.

\subsection{Two nonnegative path kernels}
All generating functions in this section are formal series in
$\R[[z]]$. Their products are defined by finite convolution of
coefficients, so no convergence assumption is needed.
For a graph \(U\), let \(A_U\) denote its adjacency matrix. The
geometric-series identity
$(I-zA_U)^{-1}=\sum_{k\ge0}A_U^kz^k$ holds for every finite graph $U$;
multiplication on either side by $I-zA_U$ verifies it coefficientwise.
A scalar series with constant term one is invertible: writing it as
$1-q$, with $q(0)=0$, its inverse is $\sum_{a\ge0}q^a$.
Each coefficient in this sum involves only finitely many terms.
Products and sums of nonnegative scalar series remain nonnegative.

For a rooted tree \(U\) with root \(r\), let \(\one\) denote the
all-one vector indexed by \(V(U)\), and define
$f_U=e_r\trans(I-zA_U)^{-1}\one=\sum_{k\geq0}f_k(U,r)z^k$. Similarly,
define $g_U=e_r\trans(I-zA_U)^{-1}e_r=\sum_{k\geq0}g_k(U,r)z^k$.
Here $f_k(U,r)$ and $g_k(U,r)$ count walks and closed walks of length
$k$ starting at $r$, respectively. For two rooted trees \(U\) and \(V\), put
$K_{U,V}=(1-2z)f_Uf_V+z(f_Ug_V+f_Vg_U)$.
In particular $g_U$ has nonnegative coefficients and constant term
one, and $K_{U,V}=K_{V,U}$.
Write $P_a$ for the path on $a$ vertices, rooted at an endpoint.

\begin{lemma}\label{lem:small-kernels}
For every finite rooted tree $U$, both $K_{P_1,U}$ and $K_{P_2,U}$ are
nonnegative.
\end{lemma}

\begin{proof}
For $P_1$, both $f$ and $g$ equal one, so
$K_{P_1,U}=(1-z)f_U+zg_U$.
If $U\neq P_1$, then $A_U\one\ge\one$, and for every $k\ge1$,
$f_k(U,r)-f_{k-1}(U,r) =e_r\trans A_U^{k-1}(A_U\one-\one)\ge0$.
Thus
$(1-z)f_U=1+\sum_{k\geq1}\bigl(f_k(U,r)-f_{k-1}(U,r)\bigr)z^k\ge0$.
The series $g_U$ counts closed walks and is nonnegative.
If $U=P_1$, the preceding kernel formula is simply one. Hence $K_{P_1,U}\geq 0$.

For $P_2$, we have $f_{P_2}=\sum_{k\geq 0}z^k=1/(1-z)$ and
$g_{P_2}=\sum_{k\geq 0}z^{2k}=1/(1-z^2)$. Substitution in the definition of $K_{U,V}$ gives
$K_{P_2,U} =\frac{(1-2z^2)f_U+z(1+z)g_U}{1-z^2}$.
If \(|U|\ge3\), every vertex \(v\) starts at least two walks of
length two. Write \(\deg_U(v)\) for the degree of \(v\) in \(U\).
Indeed, if \(\deg_U(v)\ge2\), each neighbour has degree at
least one; if $v$ is a leaf, its neighbour has degree at least two,
since $U$ is connected and has at least three vertices.
Hence $A_U^2\one\ge2\one$.
For $k\ge2$ it follows that
$f_k(U,r)-2f_{k-2}(U,r) =e_r\trans A_U^{k-2}(A_U^2\one-2\one)\ge0$.
The constant and linear coefficients of $(1-2z^2)f_U$ are $1$ and
$f_1(U,r)$, while its coefficient of $z^k$ is
$f_k(U,r)-2f_{k-2}(U,r)\ge0$ for $k\ge2$. Hence $(1-2z^2)f_U$ is
nonnegative. The series $g_U$ and $1/(1-z^2)$ are also nonnegative, so
$K_{P_2,U}\ge0$ when $|U|\ge3$.
The two smaller cases are
$K_{P_2,P_1}=1+z/(1-z^2)$ and
$K_{P_2,P_2}=(1+2z)/(1-z^2)$.
\end{proof}

\subsection{Exact root-deletion recurrences}
For a rooted tree $U$, set
$B_U=(1-2z)f_U^2-2f_U+(1+2zf_U)g_U$ and
$\mathcal S_U=\sum_{k\ge0}S_k(U,r)z^k$, and let
$\mathcal W_U=\one\trans(I-zA_U)^{-1}\one =\sum_{k\ge0}w_k(U)z^k$.
Delete the root $r$, and let $U_1,\ldots,U_s$ be the resulting
components. Each component has a unique vertex $r_i$ adjacent to
the old root: existence follows from connectedness, and two such
vertices would produce a cycle. Root $U_i$ at $r_i$, and compute
all quantities for $U_i$ in that component itself.
Put $F=\sum_i f_{U_i}$ and $G=\sum_i g_{U_i}$, and define
$\mathcal C_U=\sum_i B_{U_i}+2\sum_{i<j}K_{U_i,U_j}$.
Empty sums are zero.

\begin{lemma}\label{lem:root-recursions}
With this notation,
\[
\begin{aligned}
g_U&=(1-z^2G)^{-1}, & f_U&=g_U(1+zF),\\
\mathcal W_U&=\sum_i\mathcal W_{U_i}+\frac{f_U^2}{g_U},\\
B_U&=z^2g_U^2\mathcal C_U,\\
\mathcal S_U&=\sum_i\mathcal S_{U_i}+zg_U\mathcal C_U.
\end{aligned}
\]
\end{lemma}

\begin{proof}
Let \(R_i=(I-zA_{U_i})^{-1}\), let \(b_i=e_{r_i}\), and let
\(\one_i\) denote the all-one vector, all in the coordinates of \(U_i\).
Set $y=(I-zA_U)^{-1}e_r$, so $(I-zA_U)y=e_r$ and $y_r=g_U$. Its root
equation is $y_r-z\sum_{i=1}^{s}y_{r_i}=1$.
For each $v\in V(U_i)$, the $v$-coordinate of
$(I-zA_U)y=e_r$ is
\[
y_v-z\sum_{\substack{u\in V(U_i)\\u\sim v}}y_u
=
\begin{cases}
zy_r,&v=r_i,\\
0,&v\ne r_i,
\end{cases}
\]
because $r_i$ is the unique vertex of $U_i$ adjacent to $r$.
Writing $y_i$ for the restriction of $y$ to $V(U_i)$, these equations
combine into
$(I-zA_{U_i})y_i=zb_i y_r$.
Multiplying by $R_i=(I-zA_{U_i})^{-1}$ gives
$y_i=zR_i b_i y_r$.
The root equation consequently becomes
$y_r-z\sum_i b_i\trans y_i =y_r-z^2\sum_i g_{U_i}y_r=1$.
Combined with the definition of $G$ and $y_r=g_U$, this proves the
identity for $g_U$ in Lemma~\ref{lem:root-recursions}.

Next let $x=(I-zA_U)^{-1}\one$, so $(I-zA_U)x=\one$ and $x_r=f_U$.
Its root equation is $x_r-z\sum_{i=1}^{s}x_{r_i}=1$, and the branch
equations give $x_i=R_i\one_i+zR_i b_i x_r$. Substitution into the
root equation yields $(1-z^2G)x_r=1+zF$, proving the identity for $f_U$
in Lemma~\ref{lem:root-recursions}.
Both divisions are valid because $1-z^2G$ has constant term one.

The matrix $R_i$ is symmetric, since it is a formal series in the
symmetric matrix $A_{U_i}$. Thus
$\one_i\trans R_i b_i=b_i\trans R_i\one_i=f_{U_i}$.
Summing the coordinates of $x$ over its root and branches gives
\[
\begin{aligned}
\mathcal W_U
&=\one\trans x=x_r+\sum_i\one_i\trans x_i\\
&=f_U+\sum_i\bigl(\mathcal W_{U_i}+zf_{U_i}f_U\bigr)\\
&=\sum_i\mathcal W_{U_i}+f_U(1+zF).
\end{aligned}
\]
The $\mathcal W$-recurrence in Lemma~\ref{lem:root-recursions} follows from
$1+zF=f_U/g_U$; the constant term of $g_U$ is one.

Writing $f_i=f_{U_i}$ and $g_i=g_{U_i}$, and substituting
the definitions of $B_{U_i}$ and $K_{U_i,U_j}$ into $C_U$, the
elementary identities for $F^2$ and $FG$, together with
$f_U=g_U(1+zF)$ and $g_U^{-1}=1-z^2G$, give
\begin{align*}
F^2
&=\sum_i f_i^2+2\sum_{i<j}f_if_j,\\
FG
&=\sum_i f_ig_i+\sum_{i<j}(f_ig_j+f_jg_i),\\
\mathcal C_U
&=(1-2z)F^2-2F+G+2zFG,\\
B_U/g_U^2
&=(1-2z)(1+zF)^2+2z(1+zF)
  +\bigl(1-2(1+zF)\bigr)(1-z^2G)\\*
&=(1-2z)(1+2zF+z^2F^2)+2z+2z^2F\\*
&\hspace{2em}-1-2zF+z^2G+2z^3FG\\*
&=z^2\bigl((1-2z)F^2-2F+G+2zFG\bigr)
=z^2\mathcal C_U.
\end{align*}
This proves the $B$-recurrence in Lemma~\ref{lem:root-recursions}.

Finally, summing \eqref{eq:rooted-def} after multiplication by
$z^{k+1}$ and using $w_0(U)=|U|$ yields
$z\mathcal S_U=(1-2z)\mathcal W_U-|U|+2zf_U$. Subtracting the
corresponding identities for the branches and using
$|U|=1+\sum_i|U_i|$, the $\mathcal W$-recurrence, the identity
$f_U/g_U=1+zF$, and the $B$-recurrence in
Lemma~\ref{lem:root-recursions} gives
\begin{align*}
z\left(\mathcal S_U-\sum_i\mathcal S_{U_i}\right)
&=(1-2z)\frac{f_U^2}{g_U}-1+2z(f_U-F)\\
&=(1-2z)\frac{f_U^2}{g_U}+2zf_U+1-2\frac{f_U}{g_U}\\
&=\frac{B_U}{g_U}
=z^2g_U\mathcal C_U.
\end{align*}
Multiplication by $z$ is injective in $\R[[z]]$, as is seen by
shifting coefficient indices. Cancelling that common factor proves the
$\mathcal S$-recurrence in Lemma~\ref{lem:root-recursions}.
\end{proof}

\begin{lemma}\label{lem:endpoint-path}
For every endpoint-rooted path $P_a$, both $B_{P_a}$ and
$\mathcal S_{P_a}$ vanish.
\end{lemma}

\begin{proof}
For the single vertex, $f=g=\mathcal W=1$.
Direct substitution gives $B=0$ and
$z\mathcal S=(1-2z)-1+2z=0$, hence $\mathcal S=0$.
Suppose the result holds for the endpoint-rooted path $P_{a-1}$.
Deleting the endpoint root of $P_a$ leaves $P_{a-1}$ as its
only branch. By the induction hypothesis,
$B_{P_{a-1}}=\mathcal S_{P_{a-1}}=0$.
In the definition of $\mathcal C_{P_a}$, the sum over pairs
of distinct branches is empty, so
$\mathcal C_{P_a}=B_{P_{a-1}}=0$.
The $B$-recurrence in Lemma~\ref{lem:root-recursions} now gives $B_{P_a}=0$.
Finally, the $\mathcal S$-recurrence in Lemma~\ref{lem:root-recursions}, together with
$\mathcal S_{P_{a-1}}=0$ and $\mathcal C_{P_a}=0$,
gives $\mathcal S_{P_a}=0$.
This completes the induction.
\end{proof}

\subsection{Leaf roots of Dynkin spiders}

\begin{lemma}\label{lem:dynkin-leaf}
If $T$ is $D_n$, $E_6$, $E_7$, or $E_8$, then
$S_k(T,r)\ge0$ for every leaf $r$ and every $k\ge0$.
\end{lemma}

\begin{proof}
Fix a leaf $r$ and write its arm as
$r=v_0-v_1-\cdots-v_L=c$, where $c$ is the centre.
Let $a,b$ be the lengths of the other two arms.
In each triple in the finite Dynkin list, at most one arm length
exceeds two. Hence $\min(a,b)\le2$, regardless of the chosen leaf.

For $0\le j\le L$, let $U_j$ be the tree obtained by deleting
$v_0,\ldots,v_{j-1}$, rooted at $v_j$; in particular $U_0=T$.
The tree $U_L$ consists of the centre and the other two arms.
Its two root-deletion branches are $P_a,P_b$, rooted at their
endpoints adjacent to $c$. An arm of $a$ edges has exactly
$a$ vertices after the centre is removed.

By Lemma~\ref{lem:endpoint-path}, both branches have $B=\mathcal S=0$.
By Lemma~\ref{lem:small-kernels} and the symmetry of $K$,
$\mathcal C_{U_L}=2K_{P_a,P_b}\ge0$.
The $B$- and $\mathcal S$-recurrences in Lemma~\ref{lem:root-recursions}, together with the
nonnegative coefficients of $g_{U_L}$, imply
$B_{U_L}\ge0$ and $\mathcal S_{U_L}\ge0$.

For each $j<L$, the only branch of $U_j$ is $U_{j+1}$.
Thus $\mathcal C_{U_j}=B_{U_{j+1}}$, and the same recurrences read
$B_{U_j}=z^2g_{U_j}^2B_{U_{j+1}}$ and
$\mathcal S_{U_j}=\mathcal S_{U_{j+1}}+zg_{U_j}B_{U_{j+1}}$.
They preserve nonnegativity of both series.
Induction from $j=L$ down to $j=0$ gives $\mathcal S_T\ge0$.
\end{proof}

\subsection{A leaf attains the minimum walk count}
The following comparison applies to every spider, not only to the finite Dynkin spiders.

\begin{lemma}\label{lem:leaf-minimum}
Let $T$ be a nontrivial path or a spider. For every fixed $k\ge0$,
the minimum of $f_k(v)$ over $v\in V(T)$ is attained at a leaf.
The minimizing leaf may depend on $k$.
\end{lemma}

\begin{proof}
For a spider, we use the walk comparisons established in the proof of \cite[Theorem~12, pp.~13--14]{CZZ26}, where the rooted walk count denoted by $w_k(v)$ is our $f_k(v)$. More precisely, label the arms as $v_0v_1^i\cdots v_{t_i}^i$, where $v_0$ is the centre and $v_{t_i}^i$ is a leaf. The two comparisons obtained there from Lemma~7 and Corollary~6, and displayed in Figs.~4 and~5, are $f_k(v_s^i)\ge f_k(v_{t_i}^i)$ for every arm $i$, every $1\le s\le t_i$, and every $k\ge1$, and $f_k(v_0)\ge f_k(v_{t_i}^i)$ for every arm $i$ and every $k\ge1$.
For each vertex \(v\), apply the comparison to the leaf at the
end of an arm containing \(v\). Thus every vertex is compared with
a leaf, and the minimum is attained at a leaf.

For a path with at least three vertices, the required endpoint comparison is established in \cite[Lemma~15 and its proof, p.~819]{TWKHM13}. For a single edge, both vertices are leaves. Finally, $f_0(v)=1$ for every vertex, which covers $k=0$.
\end{proof}

\begin{proposition}\label{prop:dynkin-rooted}
Let \(T\) be a finite Dynkin tree. Then $S_k(T,r)\ge0$ for every
\(r\in V(T)\) and every $k\ge0$.
\end{proposition}

\begin{proof}
The single-vertex case was checked in Lemma~\ref{lem:endpoint-path}.
Otherwise fix $k$ and choose a minimizing leaf $r_*$ by
Lemma~\ref{lem:leaf-minimum}.
For a path, $S_k(T,r_*)=0$ by Lemma~\ref{lem:endpoint-path};
for $D_n,E_6,E_7,E_8$, it is nonnegative by
Lemma~\ref{lem:dynkin-leaf}.
The terms $w_{k+1}-2w_k$ in \eqref{eq:rooted-def} are independent
of the root, so for every $r\in V(T)$,
$S_k(T,r) =S_k(T,r_*)+2\bigl(f_k(r)-f_k(r_*)\bigr)\ge0$.
\end{proof}

\section{\texorpdfstring{ Completion of the proof and a further application}{Completion of the proof and a further application}}\label{sec:conclusion}
\subsection{Completion of the TWKHM conjecture}
\label{subsec:twkhm}
\begin{proof}[Proof of Theorem~\ref{thm:main}]
Proposition~\ref{prop:nondynkin-rooted} proves the rooted even inequality for non-Dynkin trees, while Proposition~\ref{prop:dynkin-rooted} proves it for finite Dynkin trees. Thus the rooted even inequality holds for every finite tree. Summing \eqref{eq:rooted-def} over all vertices $r\in V(T)$ and using \eqref{eq:sum-roots} gives \eqref{eq:main} for all positive even $k$.

For odd $k$, the desired inequality is the specialization $a=0$, $b=(k-1)/2$, and $c=1$ of the Sandwich Theorem \cite[Theorem~4]{TWKHM13}. We include the short spectral proof because it also identifies the equality cases. Let $u_1,\ldots,u_n$ be a real orthonormal eigenbasis of $A$, with eigenvalues $\lambda_1,\ldots,\lambda_n$, and set $c_i=(\one\trans u_i)^2\ge0$.
Then $\sum_i c_i=n$ and $w_k=\sum_i c_i\lambda_i^k$.
Expanding the right side below yields
\[
n w_{k+1}-w_1w_k
=\frac12\sum_{i,j}c_ic_j
(\lambda_i-\lambda_j)(\lambda_i^k-\lambda_j^k).
\]
For positive odd $k$, the function $x\mapsto x^k$ is strictly
increasing on $\R$. Every summand is nonnegative, so
\eqref{eq:main} follows also for odd $k$.
Equality occurs precisely when all eigenvalues with positive
weights $c_i$ are equal, or equivalently when $\one$ lies in
one eigenspace of $A$. This is equivalent to
regularity of $T$. A nontrivial regular tree has a leaf and thus
has degree one at every vertex; connectedness then makes it the
single edge. The single vertex is also regular.
Therefore the odd cases are strict when $n\ge3$.

Suppose now that $k=2m\ge2$ and equality holds in \eqref{eq:main}.
Every term in the sum \eqref{eq:sum-roots} is nonnegative, so every
$S_{2m}(T,r)$ is zero. Equation \eqref{eq:rooted-def} implies that
$f_{2m}(r)$ is independent of $r$; hence
$A^{2m}\one=c\one$
for some scalar $c\ge0$.
Since $A^2$ is symmetric positive semidefinite and $m\ge1$, the
spectral theorem gives
$\ker(A^{2m}-\mu^m I)=\ker(A^2-\mu I)$ for every $\mu\ge0$,
because the map $x\mapsto x^m$ is injective on $[0,\infty)$.
Taking $\mu=c^{1/m}$, we obtain
$A^2\one=c^{1/m}\one$.

For $n\ge2$, choose a leaf $v$ with neighbour $u$, and let $a=c^{1/m}$
be the constant coordinate of $A^2\one$.
Then $a=(A^2\one)_v=\deg_T(u)$, while
$a=(A^2\one)_u=\sum_{x\sim u}\deg_T(x)$.
The latter sum has \(\deg_T(u)\) terms, each at least one, and equals
\(\deg_T(u)\). Thus every neighbour of $u$ is a leaf.
Connectedness forces $T$ to be a star.

Conversely, for a star on $n\ge2$ vertices,
$A^2\one=(n-1)\one$. It follows that
$w_{2m}=n(n-1)^m$ and $w_{2m+1}=2(n-1)^{m+1}$ for every $m\ge0$,
which gives equality for every positive even $k$.
For $n=1$, all positive-length walk counts vanish.
For $n=2$, $A\one=\one$ and $w_k=2$ for every $k\ge0$.
This proves the equality cases for Theorem~\ref{thm:main}.
\end{proof}
\subsection{\texorpdfstring{A further application: \(U\)-restricted walk inequalities}{A further application: U-restricted walk inequalities}}
We record the following $U$-restricted walk inequality. The case $k=0$ is immediate; polynomial positivity proves the positive even-index cases, while the spectral covariance identity below proves the odd-index cases. Here $w_k(T,U)$ counts the length-$k$ walks in $T$ whose initial and terminal vertices lie in $U$; the intermediate vertices are unrestricted.
\begin{theorem}\label{thm:indicator-global-walk}
Let \(T\) be a finite connected bipartite non-Dynkin simple graph
with adjacency matrix \(A\), and let
\(\varnothing\ne U\subseteq V(T)\) be such that \(T[U]\) is a tree.
Then, for every integer \(k\ge0\),
$|U|\,w_{k+1}(T,U) - 2(|U|-1)\,w_k(T,U) \ge0$.
\end{theorem}

\begin{proof}
Since \(T[U]\) is a tree, $w_0(T,U)=|U|$ and
$w_1(T,U)=2|E(T[U])|=2(|U|-1)$.
Thus the assertion is an identity when \(k=0\).

Now let \(k=2m\) be even. For every \(r\in U\), we have
$S_{2m}(T,U,r) = w_{2m+1}(T,U)-2w_{2m}(T,U) +2f_{2m}(T,U,r)$.
Summing over \(r\in U\), and using
$\sum_{r\in U}f_{2m}(T,U,r)=w_{2m}(T,U)$,
we obtain
\[
\begin{aligned}
\sum_{r\in U}S_{2m}(T,U,r)
&=
|U|\,w_{2m+1}(T,U)
-2(|U|-1)\,w_{2m}(T,U).
\end{aligned}
\]
The left-hand side is nonnegative by Proposition
\ref{prop:indicator-rooted}, proving the even case.

It remains to consider positive odd \(k\). Write \(n:=|V(T)|\).
Let \(u_1,\ldots,u_n\) be an orthonormal eigenbasis of \(A\), with
corresponding eigenvalues \(\lambda_1,\ldots,\lambda_n\), and put
$c_i=(\one_U\trans u_i)^2\ge0$.
Then $\sum_i c_i=|U|$ and
$w_\ell(T,U)=\sum_i c_i\lambda_i^\ell$.
Consequently,
\[
\begin{aligned}
|U|\,w_{k+1}(T,U)-2(|U|-1)\,w_k(T,U)
&=w_0(T,U)w_{k+1}(T,U)-w_1(T,U)w_k(T,U)\\
&=\frac12\sum_{i,j}c_ic_j
(\lambda_i-\lambda_j)
(\lambda_i^k-\lambda_j^k).
\end{aligned}
\]
For odd \(k\), the function \(x\mapsto x^k\) is increasing on
\(\mathbb R\), so every summand is nonnegative. This proves the
odd case and hence the theorem.
\end{proof}

\section*{Statement on the use of generative AI tools}

The authors used GPT-6 Astra in the preparation of this work,
including the development of proof arguments, the checking of
intermediate steps and external theorem applications, literature
searches, and the drafting and revision of the manuscript and its
LaTeX source. All mathematical results, arguments, and proofs were independently reviewed and verified by the authors, who take full responsibility for the accuracy and content of this work.

\section*{Acknowledgements}
This work is partly supported by the National Natural Science Foundation of China (No.12371354, W2521102), the Montenegrin-Chinese Science and Technology Cooperation Project (No.4-3)  and  the Science and Technology Commission of Shanghai Municipality (No.25LN3200600).

\end{document}